\documentclass[11pt]{article}

\usepackage[margin=1in]{geometry}
\usepackage{amsmath,amssymb,amsthm,mathtools}
\usepackage{microtype}
\usepackage{aliascnt}
\usepackage[hidelinks]{hyperref}
\usepackage{cleveref}

\allowdisplaybreaks

\newtheorem{theorem}{Theorem}[section]
\newaliascnt{lemma}{theorem}
\newtheorem{lemma}[lemma]{Lemma}
\aliascntresetthe{lemma}
\newaliascnt{proposition}{theorem}
\newtheorem{proposition}[proposition]{Proposition}
\aliascntresetthe{proposition}
\newaliascnt{corollary}{theorem}
\newtheorem{corollary}[corollary]{Corollary}
\aliascntresetthe{corollary}

\newtheorem*{restatedtheorem}{Theorem}
\theoremstyle{definition}
\newaliascnt{definition}{theorem}

\aliascntresetthe{definition}
\theoremstyle{remark}
\newaliascnt{remark}{theorem}

\aliascntresetthe{remark}

\newcommand{\R}{\mathbb{R}}
\newcommand{\C}{\mathbb{C}}
\newcommand{\cE}{\mathcal{E}}
\newcommand{\norm}[1]{\left\lVert #1\right\rVert}
\newcommand{\abs}[1]{\left\lvert #1\right\rvert}
\newcommand{\seminorm}[1]{\left[\!\left[ #1\right]\!\right]}

\providecommand{\acknowledgement}[1]{\section*{Acknowledgements}#1}

\title{Convergence rates for pivoted QR and LU }
\author{Marc Aur\`ele Gilles\thanks{Department of Mathematics, Princeton University, Princeton, NJ 08544, United States. \href{mailto:gilles@princeton.edu}{gilles@princeton.edu}.}}
\date{\today}

\begin{document}

\maketitle

\begin{abstract}
Pivoted QR and pivoted LU decompositions are greedy algorithms used to compute
low-rank approximations of matrices from selected columns, or selected rows and columns. Despite their practical robustness, general worst-case bounds comparing their errors with those of the best
corresponding low-rank approximations
contain exponentially growing factors and do not explain their behavior
under modest singular value decay. We prove that under approximate greedy pivoting, their error is
controlled by the determinant of a submatrix, which is bounded by the
geometric mean of the leading singular values. Using this bound, we establish convergence rates under
algebraic and geometric singular value decay.

We also extend the LU analysis to functions of two variables. By bounding
the determinants of arbitrary sampled submatrices, we obtain algebraic
convergence rates under differentiability assumptions and
geometric convergence under analyticity. 
\end{abstract}

\section{Introduction}

Pivoted QR and LU decompositions are among the most widely used algorithms for computing low-rank approximations
of matrices, across a wide range of applications. The LU construction appears in different literatures under
names including incomplete LU, Gaussian elimination with complete pivoting,
adaptive cross approximation, and Geddes--Newton series methods
\cite{Bebendorf2000,TownsendTrefethen2013,
CarvajalChapmanGeddes2005}. Pivoted QR appears as column-pivoted QR
\cite{BusingerGolub1965} and as the greedy
algorithm in reduced-basis methods \cite{BinevEtAl2011}.

Both are greedy algorithms that build a low-rank approximation one step at
a time. Given a matrix $A\in\R^{M\times N}$, they may be written in terms
of their residuals as
\begin{equation*}
    \widehat E^{(0)}=A,
    \qquad
    \widehat E^{(k+1)}
    =
    (I-q_{k+1}q_{k+1}^\top)\widehat E^{(k)},
    \qquad
    q_{k+1}
    =
    \frac{\widehat E^{(k)}_{:,j_{k+1}}}
         {\norm{\widehat E^{(k)}_{:,j_{k+1}}}_2},
\end{equation*}
for QR, and
\begin{equation*}
    E^{(0)}=A,
    \qquad
    E^{(k+1)}
    =
    E^{(k)}
    -
    \frac{
      E^{(k)}_{:,j_{k+1}}E^{(k)}_{i_{k+1},:}
    }{
      E^{(k)}_{i_{k+1},j_{k+1}}
    },
\end{equation*}
for LU. In either case, the difference between $A$ and the residual after
$k$ steps has rank at most $k$. Greedy QR chooses the column index
$j_{k+1}$ so that
$\norm{\widehat E^{(k)}_{:,j_{k+1}}}_2$ is as large as possible, while
greedy LU chooses the location $(i_{k+1},j_{k+1})$ so that
$\abs{E^{(k)}_{i_{k+1},j_{k+1}}}$ is as large as possible.

The strong practical performance of these algorithms has resulted in their
widespread use, but it is not captured by general worst-case convergence
bounds. For example, writing $E^{(k)}$ for the residual after $k$ steps of
pivoted LU, one has
\cite{CortinovisKressnerMassei2020}\footnote{The displayed form follows
from \cite[Remark~7]{CortinovisKressnerMassei2020}.}
\begin{equation}
    \min_{0\le j\le k}\norm{E^{(j)}}_{\max}
    \le
    4^k\sigma_{k+1}(A),
    \label{eq:intro-classical-lu-bound}
\end{equation}
where
$
    \norm{A}_{\max}:=\max_{i,j}\abs{A_{ij}},
$
and $\sigma_k(A)$ denotes the $k$th singular value of $A$. A corresponding
bound for exact column-pivoted QR is
\begin{equation}
    \norm{\widehat E^{(k)}}_{2,\infty}
    \le
    2^k\sigma_{k+1}(A),
    \label{eq:intro-classical-qr-bound}
\end{equation}
where
$
    \norm{A}_{2,\infty}:=\max_j\norm{A_{:,j}}_2;
$
see \cite[Theorem~7.2]{GuEisenstat1996}.\footnote{This result is usually
stated as
$
    \norm{\widehat E^{(k)}}_2
    \le
    2^k\sqrt{N-k}\,\sigma_{k+1}(A).
$
Here $\norm{\cdot}_2$ denotes the spectral norm. The displayed
$2,\infty$-norm bound follows from the same theorem and is more directly
comparable with the bounds proved below.}
On the other hand, the Eckart--Young--Mirsky theorem states
(see, e.g., \cite{trefethen_book}) that
\[
    \sigma_{k+1}(A)
    =
    \min_{\operatorname{rank}(B)\le k}\norm{A-B}_2.
\]
Thus, the general estimates for pivoted QR and LU differ from the optimal
rank-$k$ error by exponentially growing factors. These factors are not
merely artifacts of the analysis: the factor $4^k$ in
\eqref{eq:intro-classical-lu-bound} is asymptotically sharp; see
\cite[Remark~3.3]{HarbrechtPetersSchneider2012} and
\cite[Section~6.2]{Higham1987}, and column-pivoted QR likewise admits
examples with exponentially poor error; see
\cite[Example~1]{GuEisenstat1996}.

Taken at face value, these bounds guarantee convergence only when the
singular values decay geometrically fast enough to overcome the
exponential factors. For example, if
$
    \sigma_k(A)=\mathcal O(\rho^k),
$
then \eqref{eq:intro-classical-lu-bound} guarantees convergence only when
$\rho<1/4$, while \eqref{eq:intro-classical-qr-bound} requires
$\rho<1/2$. This does not explain their behavior in practice, where
pivoted QR and LU are often observed to converge at rates close to the
singular-value decay even when that decay is much more modest.

In this paper, we derive convergence rates for pivoted QR and LU under
broad classes of singular-value decay. In particular, if
\[
    \sigma_k(A)=\mathcal O(k^{-p})
\]
for any $p>0$, then we show that
\begin{equation*}
    \min_{0\le j\le k}\norm{E^{(j)}}_{\max}
    =
    \mathcal O(k^{-p}),
    \qquad
    \norm{\widehat E^{(k)}}_{2,\infty}
    =
    \mathcal O(k^{-p}).
\end{equation*}
Thus, the best residual among the first $k$ LU iterates and the residual
of the $k$th QR iterate decay with the same algebraic exponent as the
singular values. For brevity, we refer to these two quantities as the
residuals of LU and QR, respectively. Similarly, if
$\sigma_k(A)=\mathcal O(\rho^k)$ for some $0<\rho<1$, then both residuals
converge at the geometric rate $\mathcal O(\rho^{k/2})$. These results
provide a theoretical explanation for the strong performance of the
algorithms under singular-value decay for which the usual worst-case
bounds are not informative.


A particular advantage of the LU-based approximation is that, ignoring
the pivot search, each step can be computed using only one row and one
column of the residual. Thus, the approximation can be constructed by
evaluating only a small fraction of the matrix entries. This is especially
important in the functional version of the problem, where the matrix $A$
is replaced by a function of two variables $f(x,y)$ and one seeks a
separable approximation
\begin{equation*}
    f(x,y)
    \approx
    \sum_{j=1}^k u_j(x)v_j(y).
\end{equation*}
Each step then requires only two univariate slices of the residual, which
can often be approximated accurately using a small number of function
values. This observation is the basis for algorithms including Chebfun2
\cite{TownsendTrefethen2013}, Chebfun3
\cite{HashemiTrefethen2017}, and methods based on adaptive cross
approximation \cite{Bebendorf2000,Bebendorf2011}. We extend the LU
analysis to this setting and obtain algebraic convergence rates under
differentiability assumptions and geometric convergence under
analyticity.

In both the matrix and functional settings, the dominant cost is often
the pivot search. Finding an entry of largest magnitude requires examining
the entire residual, which costs $\mathcal O(MN)$ for an
$M\times N$ matrix, while in the functional setting it requires solving
a global optimization problem. In practice, exact pivoting is therefore
often replaced by a heuristic that returns only an approximate maximum.
Our analysis allows approximate greedy pivoting and shows that it attains
the same convergence rates, up to a multiplicative constant measuring the
quality of the pivots.

The rest of the paper is organized as follows.
We first establish the matrix results in \cref{sec:matrix-case}, and then
extend the analysis to the functional setting in
\cref{sec:function-case}. We conclude in \cref{sec:conclusion}.

\section{Matrix case}
\label{sec:matrix-case}

\subsection{Background and notation}
\label{sec:background}
Let $A=[a_{ij}]\in\R^{M\times N}$. We write $A_{:,j}$ and $A_{i,:}$
for the $j$th column and $i$th row of $A$, respectively. For ordered
index tuples $I=(i_1,\ldots,i_r)$ and $J=(j_1,\ldots,j_s)$, we write
\[
    A_{I,J}
    :=
    [A_{i_\ell,j_t}]_{\ell=1,\ldots,r;\,t=1,\ldots,s}.
\]
A colon denotes all indices, so that $A_{:,J}$ and $A_{I,:}$ denote the
corresponding column and row submatrices.
We write
\[
    \norm{A}_{\max}
    :=
    \max_{1\le i\le M,\,1\le j\le N}\abs{A_{i,j}},
    \qquad
    \norm{A}_F
    :=
    \left(
        \sum_{i=1}^M\sum_{j=1}^N \abs{A_{i,j}}^2
    \right)^{1/2},
    \qquad
    \norm{A}_{2,\infty}
    :=
    \max_{1\le j\le N}\norm{A_{:,j}}_2.
\]
We denote the singular values of $A$ by
\[
    \sigma_1(A)\ge\sigma_2(A)\ge\cdots\ge0.
\]

\subsection{Geometric-mean bound for LU}
\label{subsec:matrix-lu-geometric-mean}

A pivoted LU factorization is the result of the following Gaussian elimination iteration:
\begin{equation}
\begin{aligned}
    E^{(0)}&=A, 
    & 
    E^{(k+1)} = E^{(k)}
    -
    \frac{
        E^{(k)}_{:,j_{k+1}}E^{(k)}_{i_{k+1},:}
    }{E^{(k)}_{i_{k+1},j_{k+1}}}, \\
    \widehat A^{(0)}&= 0, & 
        \widehat A^{(k+1)} = \widehat A^{(k)}
    +
    \frac{
        E^{(k)}_{:,j_{k+1}}E^{(k)}_{i_{k+1},:}
    }{E^{(k)}_{i_{k+1},j_{k+1}}} ,\\
\end{aligned}
\label{eq:matrix-update}
\end{equation}
where $\widehat A^{(k)}$ is the approximation at step $k$, and $E^{(k)}$ is the residual. 
Here $(i_{k+1},j_{k+1})$ is the pivot location and $p_{k+1}:=E^{(k)}_{i_{k+1},j_{k+1}}$ is the
pivot.
Exact greedy pivoting, also known as complete pivoting
\cite{trefethen_book}, chooses the pivot such that
\begin{equation*}
    \abs{p_{k+1}}
    =
    \norm{E^{(k)}}_{\max}.
\end{equation*}
More generally, we say that a pivoted LU factorization uses approximate greedy
pivoting if, for some fixed $0<\gamma\le1$,
\begin{equation}
    \abs{p_{k+1}}
    \ge
    \gamma\norm{E^{(k)}}_{\max}.
    \label{eq:weak-matrix-pivot}
\end{equation}
Thus, $\gamma=1$ corresponds to exact greedy pivoting.

The crucial ingredient of our analysis is the following identity.

\begin{lemma}[Pivot-product identity]
\label{lem:algorithmic-reduction}
Consider any choice of $n$ nonzero pivots $p_1,\ldots,p_n$ and
associated pivot locations
$
    (i_1,j_1),\ldots,(i_n,j_n)
$
in \eqref{eq:matrix-update}. Define
\[
    I_n=(i_1,\ldots,i_n),
    \qquad
    J_n=(j_1,\ldots,j_n),
    \qquad
    A_n:=A_{I_n,J_n}.
\]
Then
\begin{equation*}
    \det A_n=\prod_{j=1}^n p_j.
\end{equation*}
\end{lemma}

This identity is classical; it appears in Wilkinson's
analysis of Gaussian elimination with complete pivoting
\cite[Equation~(4.2)]{Wilkinson1961} and has frequently been used since then, see, e.g.,
\cite[Lemma~2]{Bebendorf2000} and
\cite[Equation~(3.1)]{Townsend2016}.
We emphasize that it holds for any sequence of nonzero pivots in
\eqref{eq:matrix-update}; no greedy pivoting assumption is required.
Combined with approximate greedy pivoting, it gives the following bound.

\begin{theorem}[Geometric-mean bound for LU]
\label{thm:matrix-convergence}

Let $A\in\R^{M\times N}$, and suppose that the first $n$ pivots in
\eqref{eq:matrix-update} are nonzero and satisfy
\eqref{eq:weak-matrix-pivot}. Then
\begin{equation*}
    \min_{0\le k\le n-1}\norm{E^{(k)}}_{\max}
    \le
    \gamma^{-1}\abs{\det A_n}^{1/n}
    \le
    \gamma^{-1}
    \left(\prod_{j=1}^n\sigma_j(A)\right)^{1/n}.
\end{equation*}
\end{theorem}

\begin{proof}
For $j=1,\ldots,n$, approximate greedy pivoting gives
\[
    \norm{E^{(j-1)}}_{\max}
    \le
    \gamma^{-1}\abs{p_j}.
\]
Therefore,
\begin{align*}
    \min_{0\le k\le n-1}\norm{E^{(k)}}_{\max}
    &\le
    \gamma^{-1}\min_{1\le j\le n}\abs{p_j}\\
    &\le
    \gamma^{-1}
    \left(\prod_{j=1}^n\abs{p_j}\right)^{1/n}\\
    &=
    \gamma^{-1}\abs{\det A_n}^{1/n},
\end{align*}
where the last equality follows from
\cref{lem:algorithmic-reduction}.

Moreover,
\[
    \abs{\det A_n}
    =
    \prod_{j=1}^n\sigma_j(A_n)
    \le
    \prod_{j=1}^n\sigma_j(A),
\]
where the inequality follows from the monotonicity of singular
values under row and column restriction~
\cite[Corollary~7.3.6]{HornJohnson2013}.
Taking $n$th roots proves the second inequality.
\end{proof}

That is, the smallest max-norm residual among the first $n$ iterates is
bounded by the geometric mean of the first $n$ singular values. This will allow us to bound the minimum residual under decay of the singular values.

The minimum over $k$ which appears in the bound may seem to require an additional search, but it is
harmless algorithmically. Indeed, under exact greedy pivoting,
$
    \abs{p_j}=\norm{E^{(j-1)}}_{\max},
$
so the smallest pivot among the first $n$ pivots identifies the smallest
residual among the corresponding iterates. Since the pivots are already
computed as part of the algorithm, one only needs to record the index of
the smallest pivot.

Under approximate greedy pivoting, the smallest pivot does not necessarily
correspond to the smallest residual. Nevertheless, the iterate 
preceding the smallest pivot satisfies the same upper bound. That is, let
$
    \ell_n
    \in
    \operatorname*{arg\,min}_{1\le j\le n}\abs{p_j}$.
Then the weak-pivot condition and
\cref{lem:algorithmic-reduction} give
\[
\begin{aligned}
    \norm{E^{(\ell_n-1)}}_{\max} \le
    \gamma^{-1}\abs{p_{\ell_n}} \le
    \gamma^{-1}
    \left(\prod_{j=1}^n\abs{p_j}\right)^{1/n}
=
    \gamma^{-1}\abs{\det A_n}^{1/n}.
\end{aligned}
\]
Thus, although this iterate may differ from the one attaining the smallest
residual, it satisfies the same determinant and singular-value bounds as
the running minimum.

\subsection{Geometric-mean bound for QR}
\label{subsec:matrix-qr-geometric-mean}

Column-pivoted QR~\cite{BusingerGolub1965} generates a sequence of
residuals
\begin{equation}
\begin{aligned}
    \widehat E^{(0)}&=A,\\
    \rho_{k+1}
    &:=
    \norm{\widehat E^{(k)}_{:,j_{k+1}}}_2,\\
    q_{k+1}
    &:=
    \frac{\widehat E^{(k)}_{:,j_{k+1}}}{\rho_{k+1}},\\
    \widehat E^{(k+1)}
    &:=
    (I-q_{k+1}q_{k+1}^\top)\widehat E^{(k)}.
\end{aligned}
\label{eq:qr-update}
\end{equation}
Here $j_{k+1}$ is the selected column index and $\rho_{k+1}$ is the norm
of the selected residual column. Writing $Q_k=[q_1\ \cdots\ q_k]$, the
residual and approximation after $k$ steps are
\[
    \widehat E^{(k)}
    =
    (I-Q_kQ_k^\top)A,
    \qquad
    \widehat A^{(k)}
    =
    Q_kQ_k^\top A.
\]
Greedy pivoting chooses $j_{k+1}$ so that
$
    \rho_{k+1}
    =
    \norm{\widehat E^{(k)}}_{2,\infty}.
$
We say that the pivoted QR factorization uses approximate greedy pivoting
if, for some fixed $0<\gamma\le1$,
\begin{equation}
    \rho_{k+1}
    \ge
    \gamma\norm{\widehat E^{(k)}}_{2,\infty}.
    \label{eq:weak-qr-pivot}
\end{equation}

The QR analogue of the pivot-product identity is the following.
\begin{lemma}[QR pivot-product identity]
\label{lem:qr-pivot-product}
Suppose that the first $n$ selected residual columns in
\eqref{eq:qr-update} are nonzero. Set
\[
    J_n=(j_1,\ldots,j_n),
    \qquad
    G_n:=A_{:,J_n}^\top A_{:,J_n}.
\]
Then
\begin{equation*}
    (\det G_n)^{1/2}
    =
    \prod_{j=1}^n\rho_j.
\end{equation*}
\end{lemma}

\begin{proof}
Write
$
    A_{:,J_n}=Q_nR_n
$, where
$R_n$ is upper triangular and satisfies
$\abs{(R_n)_{jj}}=\rho_j$; see
\cite[p.~270]{BusingerGolub1965}. Hence
\[
    G_n=R_n^\top R_n,
    \qquad
    (\det G_n)^{1/2}
    =
    \abs{\det R_n}
    =
    \prod_{j=1}^n\rho_j.
\]
\end{proof}

Combining this identity with approximate greedy pivoting gives the
corresponding geometric-mean bound.

\begin{theorem}[Geometric-mean bound for QR]
\label{thm:pivoted-qr}
Let $A\in\R^{M\times N}$, and suppose that the first $n$ selected
residual columns in \eqref{eq:qr-update} are nonzero and satisfy
\eqref{eq:weak-qr-pivot}. Then
\begin{equation*}
    \norm{\widehat E^{(n-1)}}_{2,\infty}
    \le
    \gamma^{-1}(\det G_n)^{1/(2n)}
    \le
    \gamma^{-1}
    \left(\prod_{j=1}^n\sigma_j(A)\right)^{1/n}.
\end{equation*}
\end{theorem}

\begin{proof}
Each update in \eqref{eq:qr-update} is an orthogonal projection, so
$\norm{\widehat E^{(k)}}_{2,\infty}$ is nonincreasing. Hence, for
$j=1,\ldots,n$,
\[
    \rho_j
    \ge
    \gamma\norm{\widehat E^{(j-1)}}_{2,\infty}
    \ge
    \gamma\norm{\widehat E^{(n-1)}}_{2,\infty}.
\]
Multiplying these inequalities and using
\cref{lem:qr-pivot-product} gives
\[
    \gamma^n
    \norm{\widehat E^{(n-1)}}_{2,\infty}^n
    \le
    \prod_{j=1}^n\rho_j
    =
    (\det G_n)^{1/2}.
\]
Taking $n$th roots proves the first inequality.

Finally,
\[
    (\det G_n)^{1/2}
    =
    \prod_{j=1}^n\sigma_j(A_{:,J_n})
    \le
    \prod_{j=1}^n\sigma_j(A),
\]
where the last inequality follows
from the monotonicity of singular
values under column restriction~
\cite[Corollary~7.3.6]{HornJohnson2013}.
\end{proof}

\subsection{Rates under singular-value decay}
\label{subsec:matrix-singular-value-rates}

The results for QR and LU in
\cref{thm:matrix-convergence,thm:pivoted-qr} are completely analogous,
so we give bounds for both of them at once.
Define
\[
    \varepsilon_n
    :=
    \begin{cases}
      \displaystyle
      \min_{0\le k\le n-1}\norm{E^{(k)}}_{\max},
      &\text{for LU},\\[3mm]
      \norm{\widehat E^{(n-1)}}_{2,\infty},
      &\text{for QR}.
    \end{cases}
\]
Then, in either case,
\begin{equation}
    \varepsilon_n
    \le
    \gamma^{-1}
    \left(\prod_{j=1}^n\sigma_j(A)\right)^{1/n}.
    \label{eq:common-geometric-mean-bound}
\end{equation}

\begin{corollary}[Rates under singular-value decay]
\label{cor:matrix-convergence-rates}
Under the assumptions of~\cref{thm:matrix-convergence} for LU, and
of~\cref{thm:pivoted-qr} for QR, suppose that
\[
    \sigma_j(A)\le Cj^{-p},
    \qquad
    j\ge1,
\]
for some $C>0$ and $p>0$. Then
\[
    \varepsilon_n
    \le
    \gamma^{-1}Ce^p n^{-p}.
\]
If instead
\[
    \sigma_j(A)\le C\rho^j,
    \qquad
    j\ge1,
\]
for some $C>0$ and $0<\rho<1$, then
\[
    \varepsilon_n
    \le
    \gamma^{-1}C\rho^{(n+1)/2}.
\]
\end{corollary}

\begin{proof}
By \eqref{eq:common-geometric-mean-bound}, under algebraic singular-value
decay,
\[
    \left(\prod_{j=1}^n\sigma_j(A)\right)^{1/n}
    \le
    C(n!)^{-p/n}
    \le
    Ce^p n^{-p},
\]
where we used $n!\ge(n/e)^n$. This gives the first estimate.

Under geometric singular-value decay,
\[
    \left(\prod_{j=1}^n\sigma_j(A)\right)^{1/n}
    \le
    C\rho^{(1+\cdots+n)/n}
    =
    C\rho^{(n+1)/2},
\]
which gives the second estimate.
\end{proof}

To the best of our knowledge, the LU rates in
\cref{cor:matrix-convergence-rates} are new.
However,
similar rates for QR were
already established in the reduced-basis literature in the language of Kolmogorov widths, where pivoted
QR is known as the (weak) greedy algorithm for reduced bases
\cite{BinevEtAl2011,DeVorePetrovaWojtaszczyk2013}.
In the present setting, the weak greedy algorithm for reduced bases is
applied to the finite set of columns
$
    \mathcal F_A
    :=
    \{A_{:,j}:1\le j\le N\}
    \subset\R^M.
$
The $n$th Kolmogorov width of this set is
\[
    d_n(\mathcal F_A)
    :=
    \inf_{\substack{V\subset\R^M\\ \dim V\le n}}
    \sup_{f\in\mathcal F_A}
    \inf_{v\in V}\norm{f-v}_2,
\]
where the infimum is taken over linear subspaces of $\R^M$, and its
greedy error after $k$ steps is precisely
$\norm{\widehat E^{(k)}}_{2,\infty}$.
If $U_m$ contains the first $m$ left singular vectors of $A$, then
\[
\begin{aligned}
    d_m(\mathcal F_A)
    &\le
    \max_j
    \norm{(I-U_mU_m^\top)A_{:,j}}_2
    &\le
    \norm{(I-U_mU_m^\top)A}_2
    =
    \sigma_{m+1}(A).
\end{aligned}
\]
Thus, singular-value decay implies the corresponding decay of the
Kolmogorov widths. Under $\sigma_j(A)\le Cj^{-p}$,
\cite[Corollary~3.3(ii)]{DeVorePetrovaWojtaszczyk2013} gives
\[
    \norm{\widehat E^{(k)}}_{2,\infty}
    \le
    2^{5p+1}\gamma^{-2}Ck^{-p}.
\]
In comparison, taking $n=k+1$ in
\cref{cor:matrix-convergence-rates} gives
\[
    \norm{\widehat E^{(k)}}_{2,\infty}
    \le
    Ce^p\gamma^{-1}(k+1)^{-p}.
\]
Thus, our estimates give the same algebraic rate with a smaller
explicit constant and linear, rather than quadratic, dependence on
$\gamma^{-1}$.

For geometric decay,
\cite[Corollary~3.3(i)]{DeVorePetrovaWojtaszczyk2013} shows, for $k\ge2$,
\[
\begin{aligned}
    \norm{\widehat E^{(k)}}_{2,\infty}
    &\le
    \sqrt2\,C\gamma^{-1}\rho^{(k+3)/4},
\end{aligned}
\]
when $\sigma_j(A)\le C\rho^j$. At the same iterate,
\cref{cor:matrix-convergence-rates} instead gives
\[
    \norm{\widehat E^{(k)}}_{2,\infty}
    \le
    C\gamma^{-1}\rho^{(k+2)/2}.
\]
Thus, our estimate has a sharper geometric exponent.
For exact pivoting $(\gamma=1)$, however,
\cite[Theorem~4.4]{BinevEtAl2011} provides the estimate
\[
    \norm{\widehat E^{(k)}}_{2,\infty}
    \le
    \frac{C}{\sqrt3}(2\rho)^{k+1}.
\]
This guarantees convergence only when $\rho<1/2$, but has a better asymptotic
geometric factor than our bound when $\rho<1/4$.

\section{Pivoted LU for functions}
\label{sec:function-case}

\subsection{Notation}
We now consider pivoted LU applied to functions.
Let $\mathcal X$ and $\mathcal Y$ be nonempty compact sets, and let
$f:\mathcal X\times\mathcal Y\to\R$ be continuous.
Gaussian elimination applied to $f$ generates the
iteration
\begin{equation*}
\begin{aligned}
    e_0(x,y)&:=f(x,y),\\
    e_{k+1}(x,y)
    &:=
    e_k(x,y)
    -
    \frac{
      e_k(x,y_{k+1})e_k(x_{k+1},y)
    }{
      e_k(x_{k+1},y_{k+1})
    }.
\end{aligned}
\end{equation*}
Here $(x_{k+1},y_{k+1})\in\mathcal X\times\mathcal Y$ is the pivot
location, and
\[
    p_{k+1}:=e_k(x_{k+1},y_{k+1})
\]
is the pivot. We say that the iteration uses approximate greedy pivoting
if, for some fixed $0<\gamma\le1$,
\begin{equation}
    \abs{p_{k+1}}
    \ge
    \gamma
    \norm{e_k}_{L^\infty(\mathcal X\times\mathcal Y)}.
    \label{eq:weak-pivot}
\end{equation}
Here
$\norm{g}_{L^\infty(\mathcal X\times\mathcal Y)}
:=
\sup_{(x,y)\in\mathcal X\times\mathcal Y}\abs{g(x,y)}$.
If $e_k\equiv0$, the algorithm terminates. Otherwise, continuity and
compactness imply that a
nonzero pivot satisfying \eqref{eq:weak-pivot} exists.

To obtain bounds in the function case, we will use the same determinant argument as in the matrix case. Indeed, the
pivot-product identity in \cref{lem:algorithmic-reduction} depends only on
the values of $f$ at the selected pivot coordinates. Thus, if the first
$n$ pivots are nonzero and
$A_n:=[f(x_i,y_j)]_{i,j=1}^n$, then
\[
    \det A_n=\prod_{j=1}^n p_j.
\]
Combining this identity with approximate greedy pivoting, as in the proof
of \cref{thm:matrix-convergence}, gives
\begin{equation}
    \min_{0\le k\le n-1}
    \norm{e_k}_{L^\infty(\mathcal X\times\mathcal Y)}
    \le
    \gamma^{-1}\abs{\det A_n}^{1/n}.
    \label{eq:best-residual-from-determinant}
\end{equation}

To obtain explicit convergence rates, we will bound the sampled
determinants uniformly. Before this, we note that bounds obtained this way are inherited by restrictions of the
domain. Indeed, let
$\widetilde{\mathcal X}\subset\mathcal X$ and
$\widetilde{\mathcal Y}\subset\mathcal Y$ be nonempty compact subsets.
If the algorithm is applied to the restriction of $f$ to
$\widetilde{\mathcal X}\times\widetilde{\mathcal Y}$, then
\[
\begin{aligned}
    \abs{\det A_n}^{1/n}
    &\le
    \sup_{\substack{
      x_1,\ldots,x_n\in\widetilde{\mathcal X}\\
      y_1,\ldots,y_n\in\widetilde{\mathcal Y}
    }}
    \abs{\det[f(x_i,y_j)]_{i,j=1}^n}^{1/n}\\
    &\le
    \sup_{\substack{
      x_1,\ldots,x_n\in\mathcal X\\
      y_1,\ldots,y_n\in\mathcal Y
    }}
    \abs{\det[f(x_i,y_j)]_{i,j=1}^n}^{1/n}.
\end{aligned}
\]
Thus, any determinant bound established on
$\mathcal X\times\mathcal Y$ also applies when the algorithm is run on a
restriction of the domain, and in particular on a tensor-product grid
(i.e., on a matrix).

\subsection{H\"older continuous and differentiable functions}
\label{subsec:one-dimensional-holder}

From this point on, we take
$\mathcal X=\mathcal Y=[-1,1]$. However, we note that the one-sided determinant bounds below
use no structure in the second variable and remain valid with the second
copy of $[-1,1]$ replaced by any compact set $\mathcal Y$.

We first define the regularity seminorms used below.
For a scalar function $g:[-1,1]\to\R$ and $0<\alpha\le1$, define its
$\alpha$-H\"older seminorm by
\[
    \seminorm{g}_{\alpha}
    :=
    \sup_{\substack{u,v\in [-1, 1]\\u\ne v}}
    \frac{\abs{g(u)-g(v)}}{\abs{u-v}^{\alpha}}.
\]

For $h:[-1, 1]^2\to\R$, define the uniform one-sided seminorms
\[
    \seminorm{h}_{\alpha}^{(x)}
    :=
    \sup_{y\in [-1, 1]}\seminorm{h(\cdot,y)}_{\alpha},
    \qquad
    \seminorm{h}_{\beta}^{(y)}
    :=
    \sup_{x\in [-1, 1]}\seminorm{h(x,\cdot)}_{\beta}.
\]

For $h:[-1,1]^2\to\R$ and $0<\alpha,\beta\le1$, define the mixed
H\"older seminorm by
\[
    \seminorm{h}_{\alpha,\beta}^{(\mathrm{mix})}
    :=
    \sup_{\substack{x,x'\in [-1, 1],\ x\ne x'\\
                    y,y'\in [-1, 1],\ y\ne y'}}
    \frac{
      \abs{
        h(x,y)-h(x',y)-h(x,y')+h(x',y')
      }
    }{
      \abs{x-x'}^\alpha\abs{y-y'}^\beta
    }.
\]

The integers $q_x$ and $q_y$ below denote the number of derivatives
taken in each variable.\footnote{Here and below, $\mathbb N$ includes
zero, with $\partial_x^0f=\partial_y^0f=f$.}
Thus,
$\seminorm{\partial_x^{q_x}f}_{\alpha_x}^{(x)}<\infty$
means that the $q_x$th derivative is uniformly
$\alpha_x$-H\"older in $x$. We can now state the relevant determinant bounds.

\begin{theorem}[H\"older determinant bounds]
\label{thm:integer-determinant}
Let $n\ge1$.

\begin{enumerate}
\item
Let $q_x\in\mathbb N$ and $0<\alpha_x\le1$.
Suppose that $\partial_x^{q_x}f$ exists on $[-1,1]^2$ and that
\[
    M
    :=
    \max\left\{
      \norm{f}_{L^\infty([-1,1]^2)},
      \seminorm{\partial_x^{q_x}f}_{\alpha_x}^{(x)}
    \right\}
    <\infty.
\]
Then, for arbitrary
$x_1,\ldots,x_n,y_1,\ldots,y_n\in [-1, 1]$,
\begin{equation}
    \abs{\det[f(x_i,y_j)]_{i,j=1}^n}^{1/n}
    \le
    \bigl(3(q_x+1)\bigr)^{q_x+\alpha_x}
    M n^{1/2-(q_x+\alpha_x)}.
    \label{eq:one-sided-holder-determinant}
\end{equation}

\item
Let $q_x,q_y\in\mathbb N$ and
$0<\alpha_x,\alpha_y\le1$. Suppose that
$\partial_x^{q_x}f$, $\partial_y^{q_y}f$, and
$\partial_y^{q_y}(\partial_x^{q_x}f)$ exist on $[-1,1]^2$ and that
\[
\begin{aligned}
    M
    :=
    \max\biggl\{&
      \norm{f}_{L^\infty([-1,1]^2)},
      \seminorm{\partial_x^{q_x}f}_{\alpha_x}^{(x)},
      \seminorm{\partial_y^{q_y}f}_{\alpha_y}^{(y)},\\
      &\seminorm{
        \partial_y^{q_y}\bigl(\partial_x^{q_x}f\bigr)
      }_{\alpha_x,\alpha_y}^{(\mathrm{mix})}
    \biggr\}
    <\infty.
\end{aligned}
\]
Then, for arbitrary $x_1,\ldots,x_n,y_1,\ldots,y_n\in [-1, 1]$,
\begin{equation}
\begin{aligned}
    \abs{\det[f(x_i,y_j)]_{i,j=1}^n}^{1/n}
    \le{}&
    \bigl(3(q_x+1)\bigr)^{q_x+\alpha_x}
    \bigl(3(q_y+1)\bigr)^{q_y+\alpha_y}
    M\\
    &\times
    n^{1/2-(q_x+\alpha_x)-(q_y+\alpha_y)}.
    \label{eq:integer-determinant}
\end{aligned}
\end{equation}
\end{enumerate}
\end{theorem}

The proof of \cref{thm:integer-determinant} is given in
\cref{sec:proof-differentiable-determinant}. Combining the determinant
bounds with \eqref{eq:best-residual-from-determinant} gives the
corresponding convergence rates.

\begin{corollary}[H\"older convergence rates]
\label{cor:functional-holder-rates}
 
Let $e_k$ denote the residuals generated by the LU iteration
with approximate greedy pivoting applied to $f$, and suppose that the first $n$
pivots are nonzero.

\begin{enumerate}
\item
Under the assumptions of part~1 of
\cref{thm:integer-determinant},
\[
    \min_{0\le k\le n-1}
    \norm{e_k}_{L^\infty([-1,1]^2)}
    \le
    \gamma^{-1}
    \bigl(3(q_x+1)\bigr)^{q_x+\alpha_x}
    M n^{1/2-(q_x+\alpha_x)}.
\]

\item
Under the assumptions of part~2 of
\cref{thm:integer-determinant},
\[
\begin{aligned}
    \min_{0\le k\le n-1}
    \norm{e_k}_{L^\infty([-1,1]^2)}
    \le{}&
    \gamma^{-1}
    \bigl(3(q_x+1)\bigr)^{q_x+\alpha_x}
    \bigl(3(q_y+1)\bigr)^{q_y+\alpha_y}
    M\\
    &\times
    n^{1/2-(q_x+\alpha_x)-(q_y+\alpha_y)}.
\end{aligned}
\]
\end{enumerate}
\end{corollary}

Thus, the one-sided estimate guarantees convergence whenever $f$ is merely $\alpha_x$-H\"older continuous in one variable, with
$\alpha_x>1/2$. If $f$ is Lipschitz in one variable, the one-sided
estimate gives the rate $\mathcal O(n^{-1/2})$. If $f$ satisfies the
mixed Lipschitz condition, the mixed estimate gives
$\mathcal O(n^{-3/2})$.

If $f$ is $q$ times continuously differentiable, uniformly in the other
variable, then the one-sided estimate gives the
rate
$
    \mathcal O\bigl(n^{-(q-1/2)}\bigr).
$
If, in addition, $f$ has continuous mixed derivatives up to order $q$ in
each variable, then the mixed estimate gives the improved rate
$
    \mathcal O\bigl(n^{-(2q-1/2)}\bigr).
$
We highlight that differentiability in the two variables separately gives only the
one-sided rate unless the corresponding mixed regularity is also
available.

To the best of our knowledge, these are the first general convergence
rate results for pivoted LU applied to functions that require neither
analyticity nor positive definiteness. In the positive-definite setting,
related rates are known,
see \cite{JeongTownsend2025,SantinHaasdonk2018}.

\subsection{Analytic functions}
\label{subsec:one-dimensional-analytic}

Finally, we give the corresponding convergence rates for analytic
functions. We first establish some notation.

For $\rho>1$, let $\cE_\rho$
denote the open Bernstein ellipse with foci at $\pm1$ and parameter
$\rho$.
Suppose that, for every $y\in[-1,1]$, the slice $f(\cdot,y)$ extends
to a holomorphic function
$
    \widetilde f_y:\cE_\rho\to\C.
$
Define
\[
    \mathcal A_\rho^{(x)}(f)
    :=
    \sup_{y\in[-1,1]}
    \sup_{z\in\cE_\rho}
    \abs{\widetilde f_y(z)}.
\]

We also use joint analyticity in both variables. Let
$\rho_x,\rho_y>1$, and suppose that $f$ extends to a holomorphic function
$
    \widetilde f:
    \cE_{\rho_x}\times\cE_{\rho_y}\to\C.
$
Define
\[
    \mathcal A_{\rho_x,\rho_y}^{(x,y)}(f)
    :=
    \sup_{(z,w)\in\cE_{\rho_x}\times\cE_{\rho_y}}
    \abs{\widetilde f(z,w)}.
\]
We can now state the determinant bounds.
\begin{theorem}[Analytic determinant bounds]
\label{thm:analytic-determinant}
Let $n\ge1$.
\begin{enumerate}

\item
Let $\rho>1$ and suppose that
$\mathcal A_\rho^{(x)}(f)\le M$. Then, for arbitrary
$x_1,\ldots,x_n,y_1,\ldots,y_n\in [-1, 1]$,
\begin{equation}
    \abs{\det[f(x_i,y_j)]_{i,j=1}^n}^{1/n}
    \le
    \frac{2M\sqrt n}{\sqrt{1-\rho^{-2}}}
    \rho^{-(n-1)/2}.
    \label{eq:analytic-determinant}
\end{equation}

\item
Let $\rho_x,\rho_y>1$, and suppose that
$\mathcal A_{\rho_x,\rho_y}^{(x,y)}(f)\le M$. Then, for arbitrary
$x_1,\ldots,x_n,y_1,\ldots,y_n\in [-1, 1]$,
\begin{equation}
    \abs{\det[f(x_i,y_j)]_{i,j=1}^n}^{1/n}
    \le
    \frac{4M\sqrt n}
         {\sqrt{(1-\rho_x^{-2})(1-\rho_y^{-2})}}
    (\rho_x\rho_y)^{-(n-1)/2}.
    \label{eq:mixed-analytic-determinant}
\end{equation}

\end{enumerate}
\end{theorem}
The proof is given in
\cref{sec:one-dimensional-analytic-estimates}.
Combining \cref{thm:analytic-determinant} with
\eqref{eq:best-residual-from-determinant} gives the corresponding
convergence rates.

\begin{corollary}[Analytic convergence rates]
\label{cor:functional-analytic-rates}

Let $e_k$ denote the residuals generated by the LU iteration
with approximate greedy pivoting applied to $f$, and suppose that the first $n$
pivots are nonzero.
\begin{enumerate}
\item
Under the assumptions of part~1 of
\cref{thm:analytic-determinant},
\[
    \min_{0\le k\le n-1}
    \norm{e_k}_{L^\infty([-1,1]^2)}
    \le
    \frac{2\gamma^{-1}M\sqrt n}
         {\sqrt{1-\rho^{-2}}}
    \rho^{-(n-1)/2}.
\]

\item
Under the assumptions of part~2 of
\cref{thm:analytic-determinant},
\[
    \min_{0\le k\le n-1}
    \norm{e_k}_{L^\infty([-1,1]^2)}
    \le
    \frac{4\gamma^{-1}M\sqrt n}
         {\sqrt{(1-\rho_x^{-2})(1-\rho_y^{-2})}}
    (\rho_x\rho_y)^{-(n-1)/2}.
\]
\end{enumerate}
\end{corollary}

Previous convergence results for pivoted LU applied to functions
required one-sided analyticity in a Bernstein ellipse
with parameter $\rho>4$ and yielded a geometric convergence factor
$(\rho/4)^{-n}$, see
\cite[Corollary~13]{CortinovisKressnerMassei2020},\cite[Theorem~8.1]{TownsendTrefethen2015}.
In contrast, \cref{cor:functional-analytic-rates} applies for every
$\rho>1$ and gives the geometric factor $\rho^{-n/2}$ under one-sided
analyticity, and $\rho^{-n}$ under joint analyticity when
$\rho_x=\rho_y=\rho$.

\section{Conclusion}
\label{sec:conclusion}

We established convergence rates for pivoted QR and LU under approximate
greedy pivoting, in terms of singular-value decay for matrices and
regularity for LU applied to functions. The main tool is simple:
approximate greedy pivoting, together with the pivot-product identities,
controls the residual through the determinant of a selected submatrix.
The convergence estimates then follow by bounding this determinant using
singular values in the matrix case and standard polynomial or analytic
approximation in the function case.

In upcoming work, we will extend these smoothness-based estimates for
functions to higher dimensions and, more generally, to broader function
classes using Kolmogorov widths, as in the reduced-basis literature
\cite{BinevEtAl2011}.

A surprising feature of the analysis is that it accommodates approximate
greedy pivoting seamlessly, affecting the bounds only through the
pivot-quality constant. This apparent robustness of greedy pivoting may help
explain the success of recently proposed fast low-rank approximation
algorithms. Indeed, several recent methods replace the exact pivot search by
pivoting on a sketch of the matrix
\cite{DongMartinsson2023,PearceEtAl2025,iterativeCUR,MelgaardGu2015}
and have shown great promise despite limited theoretical guarantees.
However, these methods are not directly covered by the present analysis:
sketch-based methods generally control projected row or column information,
such as row or column norms, which does not directly imply the max-norm pivot
condition analyzed here for LU. Extending the determinant bounds to such
settings is a natural direction for future work.

\acknowledgement{
AI tools were used extensively in developing the results and writing this
paper. In particular, GPT-5.6 Sol autonomously produced a proof of a version of~\cref{cor:functional-analytic-rates} and an argument
close to the proof in~\cref{cor:functional-holder-rates}, using a prompt similar to the one described in~\cite{Kerger2026}. Notably,
its initial proof contained the main ingredient used throughout the paper: a bound on the residual using the determinant of a submatrix. This insight enabled the author
to extend the argument to the remaining results in this paper, including
the matrix case, approximate greedy pivoting, and pivoted QR. The presentation is by the
author, and all results were proved or verified by the author.

The prompt used to generate the proof of \cref{cor:functional-analytic-rates} and the result returned
are available at \url{https://github.com/ma-gilles/LU_rates_proof}.
}

\appendix

\section{Proof of the H\"older determinant bound}
\label{sec:proof-differentiable-determinant}

The proof is based on polynomial interpolation: each sampled row can
be written as a linear combination of nearby rows plus a small interpolation
residual. Subtracting the linear combination preserves the determinant, so
bounding the determinant reduces to bounding a product of local
interpolation errors.

To this end, for a tuple $\mathbf t=(t_0,\ldots,t_r)$ of distinct points
of $[-1,1]$, define the Lagrange basis polynomials and the associated
interpolation and residual operators by
\[
\begin{aligned}
    \ell_i^{\mathbf t}(s)
    &:=
    \prod_{\substack{0\le j\le r\\j\ne i}}
    \frac{s-t_j}{t_i-t_j},
    \qquad i=0,\ldots,r,\\
    (\mathcal I_{\mathbf t}g)(s)
    &:=
    \sum_{i=0}^r g(t_i)\ell_i^{\mathbf t}(s),
    \qquad
    \mathcal R_{\mathbf t}g
    :=
    g-\mathcal I_{\mathbf t}g.
\end{aligned}
\]
Thus, $\mathcal I_{\mathbf t}g$ is the Lagrange interpolant of $g$ at
the nodes in $\mathbf t$, and $\mathcal R_{\mathbf t}g$ is the
corresponding residual.

The following proposition is a H\"older variant of the standard Lagrange
interpolation error estimate; see, for example,
\cite[Theorem~6.2]{SuliMayers2003}. We then give the bivariate version used
below.

\begin{proposition}[H\"older interpolation estimate]
\label{prop:local-integer-interpolation}
Let $q\in\mathbb N$, let $0<\alpha\le1$, and let
$g\in C^q([-1,1])$ satisfy
$\seminorm{g^{(q)}}_{\alpha}<\infty$. Let
\[
    t_0<\cdots<t_{q+1}
\]
be points in $[-1,1]$. Set
$\mathbf t:=(t_0,\ldots,t_q)$ and $h:=t_{q+1}-t_0$.
Then
\begin{equation*}
    \abs{(\mathcal R_{\mathbf t}g)(t_{q+1})}
    \le
    \seminorm{g^{(q)}}_{\alpha} h^{q+\alpha}.
\end{equation*}
\end{proposition}

\begin{proof}
Let $r =\mathcal R_{\mathbf t}g.
$
Since
$
    r(t_0)=\cdots=r(t_q)=0,
$
repeated application of Rolle's theorem shows that, for each
$k=0,\ldots,q$, the derivative $r^{(k)}$ has a zero
$\zeta_k\in[t_0,t_q]$. Hence, for $k=0,\ldots,q-1$, the
mean-value theorem gives
\[
\begin{aligned}
    \norm{r^{(k)}}_{L^\infty([t_0,t_{q+1}])}
    &=
    \sup_{x\in[t_0,t_{q+1}]}
    \abs{r^{(k)}(x)-r^{(k)}(\zeta_k)}\\
    &\le
    h\norm{r^{(k+1)}}_{L^\infty([t_0,t_{q+1}])}.
\end{aligned}
\]
Iterating,
\[
    \norm{r}_{L^\infty([t_0,t_{q+1}])}
    \le
    h^q
    \norm{r^{(q)}}_{L^\infty([t_0,t_{q+1}])}.
\]

Since $(\mathcal I_{\mathbf t}g)^{(q)}$ is constant, and
$r^{(q)}(\zeta_q)=0$, it follows that
\[
\begin{aligned}
    \abs{r^{(q)}(x)}
    &=
    \abs{
      g^{(q)}(x)-g^{(q)}(\zeta_q)
    }\\
    &\le
    \seminorm{g^{(q)}}_{\alpha}h^\alpha,
    \qquad x\in[t_0,t_{q+1}].
\end{aligned}
\]
Therefore,
\[
\begin{aligned}
    \abs{g(t_{q+1})-\mathcal I_{\mathbf t}g(t_{q+1})}
    &=
    \abs{r(t_{q+1})}\\
    &\le
    \seminorm{g^{(q)}}_{\alpha}h^{q+\alpha}.
\end{aligned}
\]
\end{proof}

\begin{lemma}[Bivariate H\"older interpolation estimate]
\label{lem:iterated-holder-interpolation}
Let $q_x,q_y\in\mathbb N$ and let
$0<\alpha_x,\alpha_y\le1$.
Suppose that $
    \partial_y^{q_y}\bigl(\partial_x^{q_x}f\bigr)
$
exists on $[-1,1]^2$ and satisfies
\[
    \seminorm{ \partial_y^{q_y}\bigl(\partial_x^{q_x}f\bigr)}_{\alpha_x,\alpha_y}^{(\mathrm{mix})}<\infty.
\]
Let
\[
    \xi_0<\cdots<\xi_{q_x+1},
    \qquad
    \eta_0<\cdots<\eta_{q_y+1}
\]
be points in $[-1,1]$.
Set $\boldsymbol\xi:=(\xi_0,\ldots,\xi_{q_x})$,
$\boldsymbol\eta:=(\eta_0,\ldots,\eta_{q_y})$,
$h_x:=\xi_{q_x+1}-\xi_0$, and
$h_y:=\eta_{q_y+1}-\eta_0$.
Define the residual operators in the two variables by
\[
\begin{aligned}
    (\mathcal R_x\varphi)(x,y)
    &:=
    \bigl(
      \mathcal R_{\boldsymbol\xi}(\varphi(\cdot,y))
    \bigr)(x),\\
    (\mathcal R_y\varphi)(x,y)
    &:=
    \bigl(
      \mathcal R_{\boldsymbol\eta}(\varphi(x,\cdot))
    \bigr)(y).
\end{aligned}
\]
Then
\begin{equation*}
    \abs{
      (\mathcal R_x\mathcal R_y f)
      (\xi_{q_x+1},\eta_{q_y+1})
    }
    \le
    \seminorm{
      \partial_y^{q_y}\bigl(\partial_x^{q_x}f\bigr)
    }_{\alpha_x,\alpha_y}^{(\mathrm{mix})}
    h_x^{q_x+\alpha_x}h_y^{q_y+\alpha_y}.
\end{equation*}
\end{lemma}

\begin{proof}
Set $K:=  \partial_y^{q_y}\bigl(\partial_x^{q_x}f\bigr)$.
For $u,v\in[-1,1]$, define
\[
    H_{u,v}(y)
    :=
    \partial_x^{q_x}f(u,y)
    -
    \partial_x^{q_x}f(v,y).
\]
Then
\[
    \partial_y^{q_y}H_{u,v}(y)
    =
    K(u,y)-K(v,y),
\]
and therefore
\begin{align*}
    \seminorm{\partial_y^{q_y}H_{u,v}}_{\alpha_y}
    &=
    \sup_{\substack{y,y'\in[-1,1]\\y\ne y'}}
    \frac{
      \abs{
        K(u,y)-K(v,y)-K(u,y')+K(v,y')
      }
    }{
      \abs{y-y'}^{\alpha_y}
    }\\
    &\le
    \seminorm{K}_{\alpha_x,\alpha_y}^{(\mathrm{mix})}
    \abs{u-v}^{\alpha_x}.
\end{align*}

Now set
$G(x):=(\mathcal R_yf)(x,\eta_{q_y+1})$.
Differentiating in the $x$ variable
gives
\begin{align*}
    G^{(q_x)}(u)-G^{(q_x)}(v)
    &=
    \bigl(
      \mathcal R_{\boldsymbol\eta}
      (
        \partial_x^{q_x}f(u,\cdot)
        -
        \partial_x^{q_x}f(v,\cdot)
      )
    \bigr)(\eta_{q_y+1})\\
    &=
    (\mathcal R_{\boldsymbol\eta}H_{u,v})(\eta_{q_y+1}).
\end{align*}
Applying \cref{prop:local-integer-interpolation} to $H_{u,v}$ in the
$y$ variable yields
\begin{align*}
    \abs{
      G^{(q_x)}(u)-G^{(q_x)}(v)
    }
    &\le
    \seminorm{\partial_y^{q_y}H_{u,v}}_{\alpha_y}
    h_y^{q_y+\alpha_y}\\
    &\le
    \seminorm{K}_{\alpha_x,\alpha_y}^{(\mathrm{mix})}
    \abs{u-v}^{\alpha_x}h_y^{q_y+\alpha_y}.
\end{align*}
Thus
$
    \seminorm{G^{(q_x)}}_{\alpha_x}
    \le
    \seminorm{K}_{\alpha_x,\alpha_y}^{(\mathrm{mix})}
    h_y^{q_y+\alpha_y}
$.

Finally, by the definitions of $G$ and $\mathcal R_x$,
\[
    (\mathcal R_{\boldsymbol\xi}G)(\xi_{q_x+1})
    =
    (\mathcal R_x\mathcal R_yf)
    (\xi_{q_x+1},\eta_{q_y+1}).
\]
A second application of
\cref{prop:local-integer-interpolation} in the $x$ variable gives
\begin{align*}
    \abs{
      (\mathcal R_x\mathcal R_yf)
      (\xi_{q_x+1},\eta_{q_y+1})
    }
    &\le
    \seminorm{G^{(q_x)}}_{\alpha_x}
    h_x^{q_x+\alpha_x}\\
    &\le
    \seminorm{K}_{\alpha_x,\alpha_y}^{(\mathrm{mix})}
    h_x^{q_x+\alpha_x}h_y^{q_y+\alpha_y}.
\end{align*}
\end{proof}

\begin{lemma}[Product bound]
\label{lem:consecutive-span-product}
Let $z_1<\cdots<z_n$ be points of $[-1, 1]$, let $q\ge1$
be an integer, and let $\mu>0$. Define
\[
    a_i:=
    \begin{cases}
      1,&i\le q,\\
      (z_i-z_{i-q})^\mu,&i>q.
    \end{cases}
\]
Then
\begin{equation*}
    \left(\prod_{i=1}^n a_i\right)^{1/n}
    \le
    (3q)^\mu n^{-\mu}.
\end{equation*}
\end{lemma}

\begin{proof}
If $n\le q$, then every $a_i=1$, while
$
    (3q)^\mu n^{-\mu}
    =
    \left(\frac{3q}{n}\right)^\mu
    \ge1.
$
Suppose $n>q$, and put
\[
    h_i:=z_i-z_{i-q},
    \qquad i=q+1,\ldots,n.
\]
Writing $\delta_j:=z_{j+1}-z_j$, every adjacent gap occurs in at most $q$
of the spans $h_i$; therefore,
\[
    \sum_{i=q+1}^n h_i
    \le
    q\sum_{j=1}^{n-1}\delta_j
    =
    q(z_n-z_1)
    \le
    2q.
\]
Applying the arithmetic--geometric mean inequality to the $n$ numbers
\[
    \underbrace{1,\ldots,1}_{q\text{ times}},
    h_{q+1},\ldots,h_n,
\]
we obtain
\[
    \left(\prod_{i=q+1}^n h_i\right)^{1/n}
    \le
    \frac{q+\sum_{i=q+1}^n h_i}{n}
    \le
    \frac{3q}{n}.
\]
Since $
    \prod_{i=1}^n a_i
    =
    \left(\prod_{i=q+1}^n h_i\right)^\mu,$
it follows that $
    \left(\prod_{i=1}^n a_i\right)^{1/n}
    \le
    \left(\frac{3q}{n}\right)^\mu
$.
\end{proof}

\begin{restatedtheorem}[Statement of \Cref{thm:integer-determinant}]
Let $n\ge1$.

\begin{enumerate}
\item
Let $q_x\in\mathbb N$ and $0<\alpha_x\le1$.
Suppose that $\partial_x^{q_x}f$ exists on $[-1,1]^2$ and that
\[
    M
    :=
    \max\left\{
      \norm{f}_{L^\infty([-1,1]^2)},
      \seminorm{\partial_x^{q_x}f}_{\alpha_x}^{(x)}
    \right\}
    <\infty.
\]
Then, for arbitrary
$x_1,\ldots,x_n,y_1,\ldots,y_n\in [-1, 1]$,
\[
    \abs{\det[f(x_i,y_j)]_{i,j=1}^n}^{1/n}
    \le
    \bigl(3(q_x+1)\bigr)^{q_x+\alpha_x}
    M n^{1/2-(q_x+\alpha_x)}.
\]

\item
Let $q_x,q_y\in\mathbb N$ and
$0<\alpha_x,\alpha_y\le1$. Suppose that
$\partial_x^{q_x}f$, $\partial_y^{q_y}f$, and
$\partial_y^{q_y}(\partial_x^{q_x}f)$ exist on $[-1,1]^2$ and that
\[
\begin{aligned}
    M
    :=
    \max\biggl\{&
      \norm{f}_{L^\infty([-1,1]^2)},
      \seminorm{\partial_x^{q_x}f}_{\alpha_x}^{(x)},
      \seminorm{\partial_y^{q_y}f}_{\alpha_y}^{(y)},\\
      &\seminorm{
        \partial_y^{q_y}\bigl(\partial_x^{q_x}f\bigr)
      }_{\alpha_x,\alpha_y}^{(\mathrm{mix})}
    \biggr\}
    <\infty.
\end{aligned}
\]
Then, for arbitrary
$x_1,\ldots,x_n,y_1,\ldots,y_n\in [-1, 1]$,
\[
\begin{aligned}
    \abs{\det[f(x_i,y_j)]_{i,j=1}^n}^{1/n}
    \le{}&
    \bigl(3(q_x+1)\bigr)^{q_x+\alpha_x}
    \bigl(3(q_y+1)\bigr)^{q_y+\alpha_y}
    M\\
    &\times
    n^{1/2-(q_x+\alpha_x)-(q_y+\alpha_y)}.
\end{aligned}
\]
\end{enumerate}
\end{restatedtheorem}

\begin{proof}[Proof of Theorem~\ref{thm:integer-determinant}]
Let $A:=[f(x_i,y_j)]_{i,j=1}^n$.
If two of the $x_i$ or two of the $y_j$ coincide, then two rows or
columns of $A$ coincide, so $\det A=0$. We may therefore permute the
rows and assume that $
    x_1<\cdots<x_n.
$
For the mixed estimate, also assume that $
    y_1<\cdots<y_n.
$
These permutations do not change the absolute value of the determinant.

Set $r_x:=q_x+1$. For $i>r_x$, let $\mathcal I_i^xg$ be the
Lagrange polynomial interpolating $g$ at
$x_{i-r_x},\ldots,x_{i-1}$, and set
$\mathcal R_i^xg:=g-\mathcal I_i^xg$. Then
\[
\begin{aligned}
    (\mathcal R_i^xg)(x_i)
    &=
    g(x_i)
    -
    \sum_{\ell=i-r_x}^{i-1}
    \lambda_{i\ell}^{(x)}g(x_\ell), \qquad
    \lambda_{i\ell}^{(x)}
    &:=
    \prod_{\substack{k=i-r_x\\k\ne\ell}}^{i-1}
    \frac{x_i-x_k}{x_\ell-x_k}.
\end{aligned}
\]

Define the associated linear map $\mathcal T_x:\R^n\to\R^n$ by
\[
    (\mathcal T_xv)_i
    :=
    \begin{cases}
      v_i,
      &i\le r_x,\\[1mm]
      v_i-
      \displaystyle\sum_{\ell=i-r_x}^{i-1}
      \lambda_{i\ell}^{(x)}v_\ell,
      &i>r_x.
    \end{cases}
\]
Thus, when $v_i=g(x_i)$, the map $\mathcal T_x$ leaves the first $r_x$
samples unchanged and replaces each subsequent sample by its local
interpolation residual $(\mathcal R_i^xg)(x_i)$.
Let $L_x$ be the matrix representing $\mathcal T_x$. Since
$(\mathcal T_xv)_i$ depends only on $v_1,\ldots,v_i$, and the coefficient
of $v_i$ is one, $L_x$ is unit lower triangular. Therefore,
$\det L_x=1$.

Set $A^{(x)}:=L_xA$. Its entries are
\[
    A^{(x)}_{ij}
    =
    \begin{cases}
      f(x_i,y_j),
      &i\le r_x,\\[1mm]
      (\mathcal R_i^x f)(x_i,y_j),
      &i>r_x.
    \end{cases}
\]
Define
\[
    a_i:=
    \begin{cases}
      1,&i\le r_x,\\
      (x_i-x_{i-r_x})^{q_x+\alpha_x},&i>r_x.
    \end{cases}
\]
The supremum-norm bound for $i\le r_x$ and
\cref{prop:local-integer-interpolation} for $i>r_x$ give
$\abs{A^{(x)}_{ij}}\le Ma_i$. Thus every row of $A^{(x)}$ has
Euclidean norm at most $M\sqrt n\,a_i$. Hadamard's inequality and
\cref{lem:consecutive-span-product}, applied with
$q=r_x=q_x+1$ and $\mu=q_x+\alpha_x$, give
\[
\begin{aligned}
    \abs{\det A}^{1/n}
    &=
    \abs{\det A^{(x)}}^{1/n}\\
    &\le
    M\sqrt n
    \left(\prod_{i=1}^n a_i\right)^{1/n}\\
    &\le
    \bigl(3(q_x+1)\bigr)^{q_x+\alpha_x}
    M n^{1/2-(q_x+\alpha_x)}.
\end{aligned}
\]
This proves \eqref{eq:one-sided-holder-determinant}.
No ordering or regularity in the second variable was used, so the same
argument applies with the second copy of $[-1, 1]$ replaced by any set
$\mathcal Y$.

We now prove the mixed estimate by making the same construction in the
$y$ variable. Set $r_y:=q_y+1$. For $j>r_y$, let
$\mathcal I_j^yg$ be the Lagrange polynomial interpolating $g$ at
$y_{j-r_y},\ldots,y_{j-1}$, and set $
    \mathcal R_j^yg:=g-\mathcal I_j^yg.
$
Let $L_y$ be the unit lower triangular matrix obtained from the same
construction as $L_x$, and define
\[
    b_j:=
    \begin{cases}
      1,&j\le r_y,\\
      (y_j-y_{j-r_y})^{q_y+\alpha_y},&j>r_y.
    \end{cases}
\]
Right multiplication by $L_y^\top$ applies this transformation to each
row of $A^{(x)}$. Set
$\widetilde A:=A^{(x)}L_y^\top=L_xAL_y^\top$, so
$\det\widetilde A=\det A$.

The transformed matrix has the block form
\[
    \widetilde A
    =
    \left[
    \begin{array}{c|c}
      [f(x_i,y_j)]_
      {\substack{1\le i\le r_x\\1\le j\le r_y}}
      &
      [(\mathcal R_j^y f)(x_i,y_j)]_
      {\substack{1\le i\le r_x\\r_y<j\le n}}
      \\ \hline
      [(\mathcal R_i^x f)(x_i,y_j)]_
      {\substack{r_x<i\le n\\1\le j\le r_y}}
      &
      [(\mathcal R_i^x\mathcal R_j^y f)(x_i,y_j)]_
      {\substack{r_x<i\le n\\r_y<j\le n}}
    \end{array}
    \right].
\]
The first case is bounded by the supremum norm. The two off-diagonal
cases are bounded by the corresponding one-variable interpolation
estimates, and the final case is bounded by
\cref{lem:iterated-holder-interpolation}. Consequently,
\begin{equation}
    \abs{\widetilde A_{ij}}
    \le
    Ma_ib_j,
    \qquad
    1\le i,j\le n.
    \label{eq:holder-entry-bound}
\end{equation}

Let $D_a:=\operatorname{diag}(a_1,\ldots,a_n)$,
$D_b:=\operatorname{diag}(b_1,\ldots,b_n)$, and
$B:=D_a^{-1}\widetilde A D_b^{-1}$. By
\eqref{eq:holder-entry-bound}, $\abs{B_{ij}}\le M$, so every row of $B$
has Euclidean norm at most $M\sqrt n$. Hadamard's inequality therefore
gives
\[
    \abs{\det B}\le(M\sqrt n)^n.
\]
Since $\widetilde A=D_aBD_b$ and $\det\widetilde A=\det A$,
\begin{equation}
    \abs{\det A}
    \le
    (M\sqrt n)^n
    \left(\prod_{i=1}^n a_i\right)
    \left(\prod_{j=1}^n b_j\right).
    \label{eq:holder-hadamard}
\end{equation}

Applying \cref{lem:consecutive-span-product} in the two variables gives
\[
\begin{aligned}
    \left(\prod_{i=1}^n a_i\right)^{1/n}
    &\le
    \bigl(3(q_x+1)\bigr)^{q_x+\alpha_x}
    n^{-(q_x+\alpha_x)},\\
    \left(\prod_{j=1}^n b_j\right)^{1/n}
    &\le
    \bigl(3(q_y+1)\bigr)^{q_y+\alpha_y}
    n^{-(q_y+\alpha_y)}.
\end{aligned}
\]
Taking $n$th roots in \eqref{eq:holder-hadamard} proves
\eqref{eq:integer-determinant}.
\end{proof}

\section{Proof of the analytic determinant bound}
\label{sec:one-dimensional-analytic-estimates}

The proof in the analytic case is similar to the H\"older proof. Polynomial
interpolation errors are replaced by geometrically decaying Chebyshev
coefficients, and the resulting factorization gives geometric decay of the
determinant.

We make use of the Chebyshev polynomials $T_m$ and their Joukowski
representation:
\[
    T_m(J(w))
    =
    \frac12\left(w^m+w^{-m}\right),
    \qquad
    m\in\mathbb N,
    \qquad
    J(w):=\frac12(w+w^{-1}).
\]
We also recall that the open region enclosed by the Bernstein ellipse may
be expressed as
\[
    \cE_\rho=J(A_\rho),
    \qquad
    A_\rho:=\{w\in\C:\rho^{-1}<\abs{w}<\rho\}.
\]

Let $\ell^2(\mathbb N)$ denote the space of square-summable
sequences, and set $\eta_0:=1$ and $\eta_m:=2$ for $m\ge1$.
We use weighted $\ell^2$ and bivariate variants of the classical bounds
on Chebyshev coefficients
\cite[Chapter~8, Theorem~8.1]{Trefethen2019}, which we prove below.

\begin{proposition}[Bounds on Chebyshev coefficients]
\label{prop:weighted-chebyshev}
The following statements hold.
\begin{enumerate}
\item Let $g:[-1,1]\to\C$ extend to a holomorphic function
$\widetilde g:\cE_\rho\to\C$, where $\rho>1$, and suppose
\begin{equation}
    \sup_{z\in\cE_\rho}\abs{\widetilde g(z)}\le M.
    \label{eq:weighted-chebyshev-extension-bound}
\end{equation}
Then there is $b=(b_m)_{m\ge0}\in\ell^2(\mathbb N)$ such that
\begin{equation}
    \sum_{m=0}^\infty\abs{b_m}^2\le M^2
    \label{eq:weighted-chebyshev-norm}
\end{equation}
and
\begin{equation}
    g(x)
    =
    \sum_{m=0}^\infty
    b_m\eta_m\rho^{-m}T_m(x),
    \qquad x\in [-1, 1].
    \label{eq:weighted-chebyshev-expansion}
\end{equation}
\item Let $g:[-1,1]^2\to\C$ extend to a holomorphic function
\[
    \widetilde g:
    \cE_{\rho_x}\times\cE_{\rho_y}\to\C,
\]
where $\rho_x,\rho_y>1$, and suppose
\begin{equation}
    \sup_{(z,w)\in\cE_{\rho_x}\times\cE_{\rho_y}}
    \abs{\widetilde g(z,w)}
    \le M.
    \label{eq:weighted-bivariate-chebyshev-extension-bound}
\end{equation}
Then there exist coefficients $c_{pq}\in\C$, $p,q\ge0$, such that
\begin{equation}
    \sum_{p=0}^{\infty} \sum_{q=0}^{\infty} \abs{c_{pq}}^2
    \le
    M^2
    \label{eq:weighted-bivariate-chebyshev-norm}
\end{equation}
and
\begin{equation}
    g(x,y)
    =
    \sum_{p,q\ge0}
    c_{pq}\eta_p\eta_q
    \rho_x^{-p}\rho_y^{-q}T_p(x)T_q(y),
    \qquad (x,y)\in [-1, 1]^2.
    \label{eq:weighted-bivariate-chebyshev-expansion}
\end{equation}
\end{enumerate}
Both series converge absolutely and uniformly on their respective
domains.
\end{proposition}

\begin{proof}
For the first statement, define
$
    G(w):=\widetilde g(J(w)).
$
Then $G$ is holomorphic on $A_\rho$ and satisfies
$\abs{G(w)}\le M$ there. Therefore, it has a unique Laurent expansion
\[
    G(w)=\sum_{m\in\mathbb Z}a_mw^m,
    \qquad
    w\in A_\rho.
\]
Since
$J(w)=J(w^{-1})$, uniqueness of the Laurent coefficients gives
$a_{-m}=a_m$ for all $m\in\mathbb Z$.

Let $1<r<\rho$.  Parseval's identity on $\abs{w}=r$ gives
\[
\begin{aligned}
    \frac1{2\pi}\int_0^{2\pi}
    \abs{G(re^{it})}^2\,dt
    &=
    \sum_{m\in\mathbb Z}\abs{a_m}^2r^{2m}\\
    &=
    \abs{a_0}^2+
    \sum_{m=1}^\infty
    \bigl(r^{2m}+r^{-2m}\bigr)\abs{a_m}^2
    \le M^2,
\end{aligned}
\]
where the inequality follows from
\eqref{eq:weighted-chebyshev-extension-bound}.  Dropping the nonnegative terms
$r^{-2m}\abs{a_m}^2$ gives
\[
    \abs{a_0}^2+
    \sum_{m=1}^\infty r^{2m}\abs{a_m}^2
    \le M^2.
\]
Since this estimate holds for every $1<r<\rho$, monotone convergence
as $r\uparrow\rho$ gives
\[
    \abs{a_0}^2+
    \sum_{m=1}^\infty \rho^{2m}\abs{a_m}^2
    \le M^2.
\]
Define $
    b_m:=\rho^ma_m$.
This proves \eqref{eq:weighted-chebyshev-norm}.

Since the Laurent series converges absolutely on the unit circle, it may
be grouped in the pairs $\{m,-m\}$, and hence
\[
\begin{aligned}
    G(w)
    =
    a_0+\sum_{m=1}^\infty a_m(w^m+w^{-m}) 
    =
    \sum_{m=0}^\infty
    b_m\eta_m\rho^{-m}T_m(J(w)).
\end{aligned}
\]
Since $J$ maps the unit circle onto $[-1, 1]$, this proves
\eqref{eq:weighted-chebyshev-expansion}.

For the second statement, define
$
    F(\zeta,\omega)
    :=
    \widetilde g(J(\zeta),J(\omega)).
$
Then $F$ is holomorphic on
$A_{\rho_x}\times A_{\rho_y}$ and satisfies
$\abs{F(\zeta,\omega)}\le M$ there.
The Laurent-expansion theorem for polyannuli
\cite[Theorem~1.118]{Scheidemann2023} gives a unique expansion
\[
    F(\zeta,\omega)
    =
    \sum_{j,k\in\mathbb Z}
    a_{jk}\zeta^j\omega^k.
\]
The identities
\[
    F(\zeta,\omega)
    =
    F(\zeta^{-1},\omega)
    =
    F(\zeta,\omega^{-1})
\]
and uniqueness of the Laurent coefficients imply
\[
    a_{-j,k}=a_{j,k},
    \qquad
    a_{j,-k}=a_{j,k},
    \qquad j,k\in\mathbb Z.
\]

Fix $1<r_x<\rho_x$ and $1<r_y<\rho_y$. Parseval's identity for the
Laurent expansion on the torus gives
\begin{equation*}
\begin{aligned}
    \frac1{(2\pi)^2}
    \int_0^{2\pi}\int_0^{2\pi}
    \abs{F(r_xe^{it},r_ye^{iu})}^2\,dt\,du
    =
    \sum_{j,k\in\mathbb Z}
    \abs{a_{jk}}^2r_x^{2j}r_y^{2k}
    \le M^2,
\end{aligned}
\end{equation*}
where the inequality follows from
\eqref{eq:weighted-bivariate-chebyshev-extension-bound}.
Keeping only the terms with nonnegative indices gives
\[
    \sum_{p,q\ge0}
    \abs{a_{pq}}^2r_x^{2p}r_y^{2q}
    \le M^2.
\]
Letting $r_x\uparrow\rho_x$ and $r_y\uparrow\rho_y$ yields
\[
    \sum_{p,q\ge0}
    \abs{a_{pq}}^2\rho_x^{2p}\rho_y^{2q}
    \le M^2.
\]
Define
$
    c_{pq}:=\rho_x^p\rho_y^qa_{pq}.
$
This proves \eqref{eq:weighted-bivariate-chebyshev-norm}.

By absolute convergence on the unit torus, the terms may be grouped
according to the signs of the two indices. Thus, for
$\abs{\zeta}=\abs{\omega}=1$,
\[
\begin{aligned}
    F(\zeta,\omega)
    &=
    \sum_{p,q\ge0}
    a_{pq}
    \bigl(\eta_pT_p(J(\zeta))\bigr)
    \bigl(\eta_qT_q(J(\omega))\bigr)\\
    &=
    \sum_{p,q\ge0}
    c_{pq}\eta_p\eta_q
    \rho_x^{-p}\rho_y^{-q}
    T_p(J(\zeta))T_q(J(\omega)).
\end{aligned}
\]
Since the coordinatewise Joukowski map sends the unit torus onto $[-1, 1]^2$,
this proves \eqref{eq:weighted-bivariate-chebyshev-expansion}.

Finally, since $\abs{T_m(t)}\le1$ on $[-1, 1]$, Cauchy--Schwarz gives
\[
    \sum_{m\ge0}\abs{b_m}\eta_m\rho^{-m}
    \le
    \norm{b}_{\ell^2}
    \left(\sum_{m\ge0}\eta_m^2\rho^{-2m}\right)^{1/2}
    <\infty,
\]
and
\[
\begin{aligned}
    \sum_{p,q\ge0}
    \abs{c_{pq}}\eta_p\eta_q\rho_x^{-p}\rho_y^{-q}
    &\le
    \left(\sum_{p,q\ge0}\abs{c_{pq}}^2\right)^{1/2}\\
    &\quad\times
    \left(\sum_{p\ge0}\eta_p^2\rho_x^{-2p}\right)^{1/2}
    \left(\sum_{q\ge0}\eta_q^2\rho_y^{-2q}\right)^{1/2}
    <\infty.
\end{aligned}
\]
These bounds are independent of the points in $[-1, 1]$ and $[-1, 1]^2$, so the
Weierstrass $M$-test gives absolute and uniform convergence.

\end{proof}

We also need two determinant bounds. They are stated here.
\begin{lemma}[Matrix determinant bounds]
\label{lem:analytic-finite-volume}
\leavevmode
\begin{enumerate}

\item Let \(u_0,\ldots,u_N\in\C^n\), where \(N\ge n-1\), and set
\[
    U=[u_0\ \cdots\ u_N].
\]
Then
\begin{equation}
    \det(UU^*)
    \le
    \prod_{r=0}^{n-1}
    \left(
      \sum_{m=r}^N\norm{u_m}_2^2
    \right).
    \label{eq:analytic-gram-tail}
\end{equation}

\item Let
\[
    U\in\C^{n\times p},
    \qquad
    C\in\C^{p\times q},
    \qquad
    V\in\C^{n\times q}.
\]
Then
\begin{equation}
    \abs{\det(UCV^*)}
    \le
    \left(\frac{\norm{C}_F}{\sqrt n}\right)^n
    \sqrt{\det(UU^*)\det(VV^*)}.
    \label{eq:analytic-finite-factorization}
\end{equation}
\end{enumerate}
\end{lemma}

\begin{proof}
For the first estimate, the Cauchy--Binet formula~\cite[Section~0.8.7]{HornJohnson2013} gives
\[
    \det(UU^*)
    =
    \sum_{0\le m_1<\cdots<m_n\le N}
    \abs{\det[u_{m_1}\ \cdots\ u_{m_n}]}^2.
\]
Using Hadamard's inequality, we get
\[
\begin{aligned}
    \det(UU^*)
    &\le
    \sum_{0\le m_1<\cdots<m_n\le N}
    \prod_{j=1}^n\norm{u_{m_j}}_2^2\\
    &\le
    \prod_{j=1}^n
    \left(
      \sum_{m=j-1}^N\norm{u_m}_2^2
    \right),
\end{aligned}
\]
because \(m_j\ge j-1\) in every increasing \(n\)-tuple.  This proves
\eqref{eq:analytic-gram-tail}.

For the second estimate, if \(U\) or \(V\) has rank less than \(n\),
then both sides of \eqref{eq:analytic-finite-factorization} vanish.
Thus, assume both have full row rank.
Take thin QR factorizations
\[
    U^*=Q_UR_U,
    \qquad
    V^*=Q_VR_V,
\]
where
$
    Q_U^*Q_U=Q_V^*Q_V=I_n
$, \(R_U,R_V\in\C^{n\times n}\) and
$
    UCV^*
    =
    R_U^*(Q_U^*CQ_V)R_V.
$
Moreover,
\[
    \abs{\det R_U}
    =
    \sqrt{\det(UU^*)},
    \qquad
    \abs{\det R_V}
    =
    \sqrt{\det(VV^*)}.
\]
Consequently,
\begin{equation}
\begin{aligned}
    \abs{\det(UCV^*)}
    =
    \sqrt{\det(UU^*)\det(VV^*)}\,
    \abs{\det(Q_U^*CQ_V)}.
    \label{eq:analytic-qr-factorization}
\end{aligned}
\end{equation}

Multiplication by \(Q_U^*\) or \(Q_V\) cannot increase the Frobenius
norm, so
\[
    \norm{Q_U^*CQ_V}_F\le\norm{C}_F.
\]
Finally, Hadamard's inequality followed by the arithmetic--geometric mean
gives, for every \(H\in\C^{n\times n}\),
\[
\begin{aligned}
    \abs{\det H}
    &\le
    \prod_{j=1}^n\norm{H_{j,:}}_2
    &\le
    \left(
      \frac1n\sum_{j=1}^n\norm{H_{j,:}}_2^2
    \right)^{n/2} 
    &=
    \left(\frac{\norm{H}_F}{\sqrt n}\right)^n.
\end{aligned}
\]
Apply this to \(H=Q_U^*CQ_V\) in
\eqref{eq:analytic-qr-factorization}.
\end{proof}

\begin{restatedtheorem}[Statement of \Cref{thm:analytic-determinant}]
Let $n\ge1$.
\begin{enumerate}
\item Let $\rho>1$ and suppose that
$\mathcal A_\rho^{(x)}(f)\le M$. Then, for arbitrary
$x_1,\ldots,x_n,y_1,\ldots,y_n\in [-1, 1]$,
\[
    \abs{\det[f(x_i,y_j)]_{i,j=1}^n}^{1/n}
    \le
    \frac{2M\sqrt n}{\sqrt{1-\rho^{-2}}}
    \rho^{-(n-1)/2}.
\]

\item Let $\rho_x,\rho_y>1$, and suppose that
$\mathcal A_{\rho_x,\rho_y}^{(x,y)}(f)\le M$. Then, for arbitrary
$x_1,\ldots,x_n,y_1,\ldots,y_n\in [-1, 1]$,
\[
    \abs{\det[f(x_i,y_j)]_{i,j=1}^n}^{1/n}
    \le
    \frac{4M\sqrt n}
         {\sqrt{(1-\rho_x^{-2})(1-\rho_y^{-2})}}
    (\rho_x\rho_y)^{-(n-1)/2}.
\]
\end{enumerate}
\end{restatedtheorem}

\begin{proof}[Proof of Theorem~\ref{thm:analytic-determinant}]
For the one-sided estimate, put $
    q:=\rho^{-1}.
$ For each sampled slice $
    g_j(x):=f(x,y_j),
$ 
Proposition~\ref{prop:weighted-chebyshev} gives a sequence
\(b_j\in\ell^2(\mathbb N)\) such that
$
    \norm{b_j}_{\ell^2}\le M
$
and
\[
    f(x,y_j)
    =
    \sum_{m=0}^\infty
    (b_j)_m\eta_mq^mT_m(x).
\]
For \(N\ge n-1\), define
\[
    u_m
    :=
    \eta_mq^m
    (T_m(x_1),\ldots,T_m(x_n))^T,
    \qquad
    m\ge0,
\]
and set
$
    U_N:=[u_0\ \cdots\ u_N].
$
Let \(B_N\in\C^{(N+1)\times n}\) be given by
$
    (B_N)_{m+1,j}:=(b_j)_m,
$
and set
$
    A_N:=U_NB_N.
$
By Proposition~\ref{prop:weighted-chebyshev}, for each \(i,j\),
\[
   (A_N)_{ij}
    =
    \sum_{m=0}^N
    (b_j)_m\eta_mq^mT_m(x_i)
    \longrightarrow
    f(x_i,y_j) \quad
     \text{ as } N\to\infty.
\]
Thus, with $A:=[f(x_i,y_j)]_{i,j=1}^n$, continuity of the determinant
gives $\det A_N\to\det A$.

Since
\[
    \norm{B_N}_F^2
    =
    \sum_{j=1}^n\sum_{m=0}^N\abs{(b_j)_m}^2
    \le
    nM^2,
\]
part 2 of Lemma~\ref{lem:analytic-finite-volume}, applied with
\(V=I_n\), gives
\begin{equation}
    \abs{\det A_N}
    \le
    M^n\sqrt{\det(U_NU_N^*)}.
    \label{eq:analytic-one-sided-finite}
\end{equation}

Since \(\abs{T_m(x)}\le1\) on \([-1, 1]\),
\[
    \norm{u_0}_2^2=n,
    \qquad
    \norm{u_m}_2^2\le4nq^{2m}
    \quad(m\ge1).
\]
Thus, for every \(r\ge0\),
\[
    \sum_{m=r}^\infty\norm{u_m}_2^2
    \le
    \frac{4n}{1-q^2}q^{2r}.
\]
Part 1 of Lemma~\ref{lem:analytic-finite-volume} therefore gives
\begin{equation}
\begin{aligned}
    \sqrt{\det(U_NU_N^*)}
    &\le
    \prod_{r=0}^{n-1}
    \left(
      \sum_{m=r}^\infty\norm{u_m}_2^2
    \right)^{1/2}\\
    &\le
    \left(
      \frac{2\sqrt n}{\sqrt{1-q^2}}
    \right)^n
    q^{n(n-1)/2}.
\end{aligned}
\label{eq:analytic-one-sided-volume}
\end{equation}
Combining this with \eqref{eq:analytic-one-sided-finite} and letting
\(N\to\infty\) yields
\[
    \abs{\det A}
    \le
    \left(
      \frac{2M\sqrt n}{\sqrt{1-q^2}}
    \right)^n
    q^{n(n-1)/2}.
\]
Taking \(n\)th roots and substituting \(q=\rho^{-1}\) proves
\eqref{eq:analytic-determinant}.

For the mixed estimate, put $
    q_x:=\rho_x^{-1},
    q_y:=\rho_y^{-1}.
$ 
Proposition~\ref{prop:weighted-chebyshev} gives coefficients
\((c_{pq})_{p,q\ge0}\) such that
$
    \sum_{p,q\ge0}\abs{c_{pq}}^2\le M^2
$
and
\[
    f(x,y)
    =
    \sum_{p,q\ge0}
    c_{pq}\eta_p\eta_q
    q_x^pq_y^qT_p(x)T_q(y).
\]

For \(N\ge n-1\), define
\[
    u_p
    :=
    \eta_pq_x^p
    (T_p(x_1),\ldots,T_p(x_n))^T,
    \qquad
    p\ge0,
\]
and
\[
    v_q
    :=
    \eta_qq_y^q
    (T_q(y_1),\ldots,T_q(y_n))^T,
    \qquad
    q\ge0.
\]
Set
\[
    U_N:=[u_0\ \cdots\ u_N],
    \qquad
    V_N:=[v_0\ \cdots\ v_N],
    \qquad
    C_N:=[c_{pq}]_{p,q=0}^N, 
    \qquad
    A_N:=U_NC_NV_N^*.
\]
By the absolute and uniform convergence in
Proposition~\ref{prop:weighted-chebyshev}, for each \(i,j\),
\[
    (A_N)_{ij}
    =
    \sum_{p,q=0}^N
    c_{pq}\eta_p\eta_q
    q_x^pq_y^qT_p(x_i)T_q(y_j)
    \longrightarrow
    f(x_i,y_j) \qquad
     \text{ as } N\to\infty.
\]
Hence
$
    A_N\longrightarrow
    A:=[f(x_i,y_j)]_{i,j=1}^n
$
entrywise and
$
    \det A_N\longrightarrow\det A.
$

Moreover,
\[
    \norm{C_N}_F^2
    =
    \sum_{p,q=0}^N\abs{c_{pq}}^2
    \le M^2.
\]
Part 2 of Lemma~\ref{lem:analytic-finite-volume} gives
\begin{equation}
\begin{aligned}
    \abs{\det A_N}
    \le
    \left(\frac{M}{\sqrt n}\right)^n
    \sqrt{
      \det(U_NU_N^*)\det(V_NV_N^*)
    }.
    \label{eq:analytic-mixed-finite}
\end{aligned}
\end{equation}

The same calculation as in \cref{eq:analytic-one-sided-volume} gives
\[
    \sqrt{\det(U_NU_N^*)}
    \le
    \left(
      \frac{2\sqrt n}{\sqrt{1-q_x^2}}
    \right)^n
    q_x^{n(n-1)/2},
\]
and
\[
    \sqrt{\det(V_NV_N^*)}
    \le
    \left(
      \frac{2\sqrt n}{\sqrt{1-q_y^2}}
    \right)^n
    q_y^{n(n-1)/2}.
\]
Substituting these bounds into
\eqref{eq:analytic-mixed-finite} yields
\[
    \abs{\det A_N}
    \le
    \left(
      \frac{
        4M\sqrt n
      }{
        \sqrt{(1-q_x^2)(1-q_y^2)}
      }
    \right)^n
    (q_xq_y)^{n(n-1)/2}.
\]
Letting \(N\to\infty\), taking \(n\)th roots, and substituting
\(q_x=\rho_x^{-1}\), \(q_y=\rho_y^{-1}\), proves
\eqref{eq:mixed-analytic-determinant}.
\end{proof}

\end{document}